\pdfoutput=1
\RequirePackage{ifpdf}
\ifpdf % We are running pdfTeX in pdf mode
\documentclass[pdftex]{sigma}
\else
\documentclass{sigma}
\fi

\newcommand{\R}{\mathbb{R}}

\newcommand{\C}{\mathbb{C}}
\newcommand{\Z}{\mathbb{Z}}

\renewcommand{\Z}{\mathcal{Z}}

\numberwithin{equation}{section}

\newtheorem{Theorem}{Theorem}[section]
\newtheorem{Lemma}[Theorem]{Lemma}
{ \theoremstyle{definition}
\newtheorem{Remarks}[Theorem]{Remarks} }

\begin{document}

%\allowdisplaybreaks

\renewcommand{\thefootnote}{$\star$}

\newcommand{\arXivNumber}{1601.02263}

\renewcommand{\PaperNumber}{046}

\FirstPageHeading

\ShortArticleName{The Asymptotic Expansion of Kummer Functions}

\ArticleName{The Asymptotic Expansion of Kummer Functions\\ for Large Values of the $\boldsymbol{a}$-Parameter,\\ and Remarks on a Paper by Olver\footnote{This paper is a~contribution to the Special Issue
on Orthogonal Polynomials, Special Functions and Applications.
The full collection is available at \href{http://www.emis.de/journals/SIGMA/OPSFA2015.html}{http://www.emis.de/journals/SIGMA/OPSFA2015.html}}}

\Author{Hans VOLKMER}

\AuthorNameForHeading{H.~Volkmer}

\Address{Department of Mathematical Sciences, University of Wisconsin-Milwaukee,\\ P.O.~Box 413, Milwaukee, WI, 53201, USA}
\Email{\href{mailto:volkmer@uwm.edu}{volkmer@uwm.edu}}

\ArticleDates{Received January 10, 2016, in f\/inal form May 01, 2016; Published online May 06, 2016}

\Abstract{It is shown that a known asymptotic expansion of the Kummer function $U(a,b,z)$ as $a$ tends to inf\/inity is valid for $z$ on the full Riemann surface of the logarithm. A corresponding result is also proved in a more general setting considered by Olver (1956).}


\Classification{33B20; 33C15; 41A60}

\renewcommand{\thefootnote}{\arabic{footnote}}
\setcounter{footnote}{0}

\section{Introduction}
Recently, the author collaborated on a project \cite{C} investigating the maximal domain in which an integral addition theorem for the Kummer function $U(a,b,z)$ due to Magnus \cite{Magnus41,Magnus} is valid.
In this work it is important to know the asymptotic expansion of $U(a,b,z)$ as $a$~tends to inf\/inity. Such an expansion is well-known, and, for instance, can be found in Slater's book~\cite{S}. Slater's expansion is in terms of modif\/ied Bessel functions $K_\nu(z)$, and it is derived from a paper by Olver \cite{O}.
However, there are two problems when we try to use the known result. As Temme~\cite{T} pointed out, there is an error
in Slater's expansion. Moreover, in all known results the range of validity for the variable $z$ is restricted to certain sectors in the $z$-plane.

The purpose of this paper is two-fold. Firstly, we correct the error in \cite{S}, and we show that the corrected expansion based on \cite{O} agrees with the result in \cite{T} which was obtained in an entirely dif\/ferent way.
Secondly, we show that the asymptotic expansion of $U(a,b,z)$ as~$a$~tends to inf\/inity is valid for $z$ on the full Riemann surface of the logarithm.
This is somewhat surprising because often the range of validity of asymptotic expansions is restricted by Stokes' lines. Olver's
results in~\cite{O} are valid for a more general class of functions (containing conf\/luent hypergeometric functions as a special case.)
He introduces a restriction on~$\arg z$, and on~\mbox{\cite[p.~76]{O}} he writes ``In the case of the series with the basis function~$K_\mu$ we establish the asymptotic property in the range $|\arg z|\le \frac32 \pi$. It is, in fact, unlikely that the valid range exceeds this $\dots$''.
However, we show in this paper that the restriction $|\arg z|\le \frac32 \pi$ can be removed at least under an additional assumption \eqref{1:D}.

In Section \ref{Olver} of this paper we review the results that we need from Olver \cite{O}. We discuss these results in Section~\ref{discussion}. In Section~\ref{removal} we prove that
Olver's asymptotic expansion holds on the full Riemann surface of the logarithm. Sections~\ref{extension}, \ref{AB} and \ref{W3} deal with extensions
to more general values of parameters. In Section~\ref{confluent} we specialize to asymptotic expansions of Kummer functions. In Section~\ref{Temme}
we make the connection to Temme~\cite{T}.

\section{Olver's work}\label{Olver}
Olver \cite[(7.3)]{O} considers the dif\/ferential equation
\begin{gather}\label{1:eq1}
w''(z)=\frac1z w'(z)+\left(u^2+\frac{\mu^2-1}{z^2}+f(z)\right)w(z) .
\end{gather}
The function $f(z)$ is even and analytic in a simply-connected domain $D$ containing $0$.
It is assumed that $\Re\mu\ge 0$.
The goal is to f\/ind the asymptotic
behavior of solutions of \eqref{1:eq1} as $0<u\to\infty$.

Olver \cite[(7.4)]{O} starts with a formal solution to \eqref{1:eq1} of the form
\begin{gather*}%\label{1:eq2}
 w(z)=z\Z_\mu(uz)\sum_{s=0}^\infty \frac{A_s(z)}{u^{2s}}+\frac{z}{u}\Z_{\mu+1}(uz)\sum_{s=0}^\infty \frac{B_s(z)}{u^{2s}} ,
\end{gather*}
where either $\Z_\mu=I_\mu$, $\Z_{\mu+1}=I_{\mu+1}$ or $\Z_\mu=K_\mu$, $\Z_{\mu+1}=-K_{\mu+1}$ are modif\/ied Bessel functions.
The functions $A_s(z)=A_s(\mu,z)$, $B_s(z)=B_s(\mu,z)$ are def\/ined by $A_0(z)=1$, and then recursively, for $s\ge 0$,
\begin{gather}
2B_s(z) = -A_s'(z)+\int_0^z\left(f(t)A_s(t)-\frac{2\mu+1}{t} A_s'(t)\right) dt,\label{1:eq3}\\
2A_{s+1}(z) = \frac{2\mu+1}{z} B_s(z)- B_s'(z)+\int f(z)B_s(z) dz.\label{1:eq4}
\end{gather}
The integral in \eqref{1:eq4} denotes an arbitrary antiderivative of $f(z)B_s(z)$.
The func\-tions~$A_s(z)$,~$B_s(z)$ are analytic in $D$, and they are even and odd, respectively.

If the domain $D$ is unbounded, Olver \cite[p.~77]{O} requires that $f(z)=O(|z|^{-1-\alpha})$ as $|z|\to\infty$, where $\alpha>0$.
In our application to the conf\/luent hypergeometric equation in Section~\ref{confluent} the function $f(z)=z^2$ does not satisfy this condition.
Therefore, throughout this paper, we will take
\begin{gather}\label{1:D}
 D=\{z\colon |z|<R_0\},
 \end{gather}
where $R_0$ is a positive constant.
Olver \cite[p.~77]{O} introduces various subdomains $D'$, $D_1$, $D_2$ of~$D$.
We may choose $D'=\{z\colon |z|\le R\}$, where $0<R<R_0$. The domain $D_1$ comprises those points~$z$ in~$D'$ which can be joined to the origin by a contour
which lies in~$D'$ and does not cross either the imaginary axis, or the line through~$z$ parallel to the imaginary axis.
For our special~$D'$ the contour can be taken as the line segment connecting $z$ and $0$, so $D_1=D'$. The domain~$D_1$ appears in
Olver \cite[Theorem~D(i)]{O}. According to this theorem, \eqref{1:eq1} has a solu\-tion~$W_1(u,z)$ of the form
\begin{gather}\label{1:W1}
W_1(u,z)=zI_\mu(uz)\left(\sum_{s=0}^{N-1} \frac{A_s(z)}{u^{2s}}+g_1(u,z)\right)
+\frac{z}{u}I_{\mu+1}(uz)\left(\sum_{s=0}^{N-1} \frac{B_s(z)}{u^{2s}}+zh_1(u,z)\right),\!\!\!
\end{gather}
where
\begin{gather}\label{1:W1a}
 |g_1(u,z)|+|h_1(u,z)|\le \frac{K_1}{u^{2N}}\qquad \text{for} \quad 0<|z|\le R,\quad u\ge u_1 .
\end{gather}

\begin{Remarks}\quad
\begin{enumerate}\itemsep=0pt
\item
The parameter $\mu$ is considered f\/ixed. We may write $W_1(u,\mu,z)$ to indicate the dependence of $W_1$ on $\mu$.
\item
Every solution $w(z)$ of \eqref{1:eq1} is def\/ined on the Riemann surface of the logarithm over~$D$. Note that there is no restriction on $\arg z$ in~\eqref{1:W1a}, see \cite[p.~76]{O}.
\item
The precise statement is this: for every positive integer $N$ there are functions~$g_1$,~$h_1$ and positive constants~$K_1$, $u_1$ (independent of $u,z$)
such that~\eqref{1:W1},~\eqref{1:W1a} hold.
\item
The functions $A_s(z)$, $B_s(z)$ are not uniquely determined because of the free choice of
integration constants in~\eqref{1:eq4}. Even if we make a def\/inite choice of these integration constants, the solution $W_1(u,z)$
is not uniquely determined by~\eqref{1:W1},~\eqref{1:W1a}. For example, one can replace $W_1(u,z)$ by $(1+e^{-u})W_1(u,z)$.
\item
Olver's construction of $W_1(u,z)$ is independent of $N$ but may depend on~$R$.
In our application to the conf\/luent hypergeometric dif\/ferential equation we have $f(z)=z^2$. Then~$R$ can be any positive number
but $W_1(u,z)$ may depend on the choice of $R$.
\item
Olver has the term $\frac{z}{1+|z|}$ in place of~$z$ in front of~$h_1$ in~\eqref{1:W1} but since we assume $|z|\le R$ this makes no dif\/ference.
\end{enumerate}
\end{Remarks}

For the def\/inition of $D_2$ we suppose that $a$ is an arbitrary point of the sector $|\arg a|<\frac12\pi$ and $\epsilon>0$.
Then $D_2$ comprises those points $z\in D'$ for which $|\arg z|\le\frac32\pi$, $\Re z\le \Re a$, and a~contour can be found joining~$z$ and~$a$
which satisf\/ies the following conditions:
\begin{itemize}\itemsep=0pt
\item[(i)] it lies in $D'$,
\item[(ii)]
it lies wholly to the right of the line through $z$ parallel to the imaginary axis,
\item[(iii)]
it does not cross the negative imaginary axis if $\frac12\pi\le \arg z\le \frac32\pi$, and does not cross
the positive imaginary axis if $-\frac32\pi\le \arg z\le -\frac\pi2$,
\item[(iv)]
it lies outside the circle $|t|=\epsilon |z|$.
\end{itemize}

In our special case $D'=\{z\colon |z|\le R\}$ we choose $a=R$. If $0\le\arg z\le \frac32\pi$ and $0<|z|\le R$, we choose the contour starting at~$z$
moving in positive direction parallel to the imaginary axis until we hit the circle $|t|=R$. Then we move clockwise along the circle $|t|=R$ towards~$a$. Taking into account condition~(iv), we see that~$D_2$ is the set of points $z$ with $-\frac32\pi+\delta\le z\le \frac32\pi-\delta$,
$0<|z|\le R$, where $\delta>0$.
The domain $D_2$ appears in Olver \cite[Theorem~D(ii)]{O}. According to this theorem, \eqref{1:eq1} has a~solution $W_2(u,z)$ of the form
\begin{gather}
W_2(u,z)=zK_\mu(uz)\left(\sum_{s=0}^{N-1} \frac{A_s(z)}{u^{2s}}+g_2(u,z)\right)\nonumber\\
\hphantom{W_2(u,z)=}{} -\frac{z}{u}K_{\mu+1}(uz)\left(\sum_{s=0}^{N-1} \frac{B_s(z)}{u^{2s}}+zh_2(u,z)\right),\label{1:W2}
\end{gather}
where
\begin{gather}\label{1:W2a}
 |g_2(u,z)|+|h_2(u,z)|\le \frac{K_2}{u^{2N}}\qquad \text{for} \quad 0<|z|\le R, \quad |\arg z|\le \frac32\pi-\delta, \quad u\ge u_2 .
\end{gather}
Note that in \eqref{1:W2a} there is a restriction on $\arg z$.

In the rest of this paper we choose the functions $A_s(z)$ such that
\begin{gather}\label{1:As}
A_s(0)=0\qquad\text{if} \quad s\ge 1.
\end{gather}
Then the functions $A_s(z)$, $B_s(z)$ are uniquely determined.

\section[Properties of solutions $W_1$ and $W_2$]{Properties of solutions $\boldsymbol{W_1}$ and $\boldsymbol{W_2}$}\label{discussion}

The dif\/ferential equation \eqref{1:eq1} has a regular singularity at $z=0$ with exponents $1\pm \mu$.
Substituting $x=z^2$ we obtain an equation which has a regular singularity at $x=0$ with exponents $\frac12(1\pm\mu)$.
Therefore, for every $\mu$ which is not a negative integer, \eqref{1:eq1} has a unique solution $W_+(z)=W_+(u,\mu,z)$ of the form
\begin{gather*}%\label{2:eq1}
W_+(z)= z^{1+\mu}\sum_{n=0}^\infty c_nz^{2n},
\end{gather*}
where the $c_n$ are determined by $c_0=1$, and
\begin{gather*}
4n(\mu+n)c_n=u^2c_{n-1}+\sum_{j=0}^{n-1}f_jc_{n-1-j} \qquad\text{for} \quad n\ge 1
\end{gather*}
when
\begin{gather*}
f(z)=\sum_{n=0}^\infty f_n z^{2n} .
\end{gather*}

If $\mu$ is not an integer, then $W_+(u,\mu,z)$ and $W_+(u,-\mu,z)$ form a fundamental system of solutions of~\eqref{1:eq1}.
If $\Re \mu\ge 0$, there is a solution $W_-(z)$ linearly independent of $W_+(z)$ such that
\begin{gather*}%\label{2:eq2}
W_-(z)=z^{1-\mu}p\big(z^2\big)+ d \ln z W_+(z),
\end{gather*}
where $p$ is a power series and $d$ is a suitable constant. If $\mu\ne 0$ we choose $p(0)=1$. If~$\mu$ is not an integer then $d=0$.

\begin{Lemma}\label{2:l1}
Suppose $\Re\mu\ge 0$.
There is a function $\alpha(u)$ such that
\begin{gather*}%\label{2:connectW1}
W_1(u,z)=\alpha(u)W_+(u,z),
\end{gather*}
and, for every $N=1,2,3,\dots$,
\begin{gather}\label{2:alpha}
\alpha(u)=\frac{2^{-\mu}u^\mu}{\Gamma(\mu+1)}\left(1+O\left(\frac1{u^{2N}}\right)\right)\qquad\text{as} \quad 0<u\to\infty.
\end{gather}
\end{Lemma}

\begin{proof}
There are functions $\alpha_+(u)$, $\alpha_-(u)$ such that
\begin{gather}\label{2:connectW1a}
W_1(u,z)=\alpha_+(u)W_+(u,z)+\alpha_-(u)W_-(u,z).
\end{gather}
Suppose $\Re\mu>0$. Then \eqref{2:connectW1a} implies
\begin{gather}\label{2:limit1}
\lim_{z\to0^+} z^{\mu-1}W_1(u,z)=\alpha_-(u).
\end{gather}
We use \cite[(10.30.1)]{NIST}
\begin{gather*}%\label{2:limitI}
 \lim_{z\to0} I_\nu(z)z^{-\nu}=\frac{2^{-\nu}}{\Gamma(\nu+1)}.
\end{gather*}
Then \eqref{1:W1}, \eqref{1:W1a} give{\samepage
\begin{gather}\label{2:limit2}
\lim_{z\to0^+} z^{\mu-1}W_1(u,z)= 0.
\end{gather}
It follows from \eqref{2:limit1}, \eqref{2:limit2} that $\alpha_-(u)=0$.}

Now suppose that $\Re\mu=0$, $\mu\ne 0$. Then we argue as before but instead of $z\to0^+$ we approach $0$ along a spiral $z=re^{\pm i r}$, $0<r\to 0$, when $\pm \Im \mu>0$.
Then along this spiral $z^{2\mu}\to 0$. We obtain again that $\alpha_-(u)=0$.
In a similar way, we also show that $\alpha_-(u)=0$ when $\mu=0$.

Therefore, \eqref{2:connectW1a} gives
\begin{gather*} \lim_{z\to0^+} z^{-\mu-1} W_1(z,u)=\alpha_+(u)\end{gather*}
and, from \eqref{1:W1}, \eqref{1:W1a}, \eqref{1:As}
\begin{gather*} \lim_{z\to0^+} z^{-\mu-1}W_1(u,z)=\frac{2^{-\mu}u^\mu}{\Gamma(\mu+1)}\left(1+O\left(\frac1{u^{2N}}\right)\right)\end{gather*}
which implies \eqref{2:alpha} with $\alpha(u)=\alpha_+(u)$.
\end{proof}

Let us def\/ine
\begin{gather*}%\label{2:W3}
W_3(u,\mu,z)= \frac{2^{-\mu}u^\mu}{\Gamma(\mu+1)}W_+(u,\mu,z) .
\end{gather*}
Then Lemma~\ref{2:l1} gives
\begin{gather*}%\label{2:connectW3}
 W_3(u,z)=\tilde\alpha(u) W_1(u,z),\qquad \text{where} \quad \tilde\alpha(u)=1+O\left(\frac1{u^{2N}}\right).
\end{gather*}
Therefore, $W_3$ admits the asymptotic expansion \eqref{1:W1}, \eqref{1:W1a}, so we can replace $W_1$ by $W_3$.
Note that in contrast to $W_1$, $W_3$ is a uniquely def\/ined function which is identif\/ied as a (Floquet) solution of~\eqref{1:eq1} and not by its asymptotic behavior as $u\to\infty$.

Unfortunately, it seems impossible to replace $W_2$ by an easily identif\/iable solution of~\eqref{1:eq1}. However, we will now prove several useful properties of~$W_2$.

\begin{Lemma}\label{2:l2}
Suppose that $\Re\mu\ge 0$. There is a function $\beta(u)$ such that
\begin{gather}\label{2:eq3}
 W_2\big(u,ze^{\pi i}\big)-e^{\pi i(1-\mu)}W_2(u,z)= \beta(u) W_3(u,z) ,
\end{gather}
and, for every $N=1,2,3,\dots$,
\begin{gather}\label{2:beta}
 \beta(u)=\pi i\left(1+O\left(\frac{1}{u^{2N}}\right)\right)\qquad\text{as}\quad 0<u\to\infty.
 \end{gather}
\end{Lemma}
\begin{proof}
We set $\lambda_\pm=e^{\pi i(1\pm\mu)}$. Equation \eqref{1:eq1} has a fundamental system of solutions~$W_+$,~$W_-$ such that
\begin{gather*}
W_+\big(ze^{\pi i}\big)=\lambda_+ W_+(z),\qquad
W_-\big(ze^{\pi i}\big)=\lambda_- W_-(z)+\rho W_+(z) .
\end{gather*}
Let $w(z)=c_+W_+(z)+c_-W_-(z)$ be any solution of \eqref{1:eq1}.
Then
\begin{gather*}
w\big(ze^{\pi i}\big)-\lambda_- w(z)=((\lambda_+-\lambda_-)c_++\rho c_-) W_+(z).
\end{gather*}
If we apply this result to $w=W_2$ we see that there is a function~$\beta(u)$ such that~\eqref{2:eq3} holds.

Let $z>0$ and set $z_1=ze^{\pi i}$.
We use~\eqref{1:W2} for $z_1$ in place of~$z$, and \cite[(10.34.2)]{NIST}
\begin{gather}\label{3:ancontK}
K_\nu\big(ze^{\pi im}\big)= e^{-\pi i \nu m}K_\nu(z)-\pi i \frac{\sin(\pi \nu m)}{\sin(\pi \nu)} I_\nu(z)
\end{gather}
with $m=1$. Then
\begin{gather*}
 W_2(u,z_1) = z\left(\lambda_-K_\mu(uz)+\pi i I_\mu(uz)\right)\left(\sum_{s=0}^{N-1}\frac{A_s(z)}{u^{2s}}+g_2(u,z_1)\right)\\
\hphantom{W_2(u,z_1) =}{} +\frac{z}{u}\left(-\lambda_-K_{\mu+1}(uz)+\pi i I_{\mu+1}(uz)\right) \left(\sum_{s=0}^{N-1}\frac{B_s(z)}{u^{2s}}+zh_2(u,z_1)\right).
\end{gather*}
Using \eqref{1:W2} a second time, we f\/ind that
\begin{gather*}
 W_2(u,z_1)-\lambda_-W_2(u,z)=\pi i zI_\mu(uz)\left(\sum_{s=0}^{N-1}\frac{A_s(z)}{u^{2s}}+g_2(u,z_1)\right)\\
\qquad{}
+\pi i\frac{z}{u}I_{\mu+1}(uz) \left(\sum_{s=0}^{N-1}\frac{B_s(z)}{u^{2s}}+zh_2(u,z_1)\right)\\
\qquad{}
+\lambda_- zK_\mu(uz) (g_2(u,z_1)-g_2(u,z))-\lambda_-\frac{z^2}{u}K_{\mu+1}(uz)(h_2(u,z_1)-h_2(u,z)).
\end{gather*}
We now expand the right-hand side of \eqref{2:eq3} using \eqref{1:W1}, and compare the expansions.
Setting $z=R$ and dividing by $R I_\mu(uR)$, we obtain
\begin{gather*}
 (\beta(u)-\pi i) \left(1+O\left(\frac1u\right)\right)=O\left(\frac{1}{u^{2N}}\right) \quad \text{as $0<u\to\infty$,}
\end{gather*}
where we used \cite[(10.40.1)]{NIST}
\begin{gather}\label{2:asyI}
 I_\nu(x)=\frac{e^x}{\sqrt{2\pi x}}\left(1+O\left(\frac1x\right)\right)\qquad\text{as} \quad 0<x\to\infty,
\end{gather}
and \cite[(10.40.2)]{NIST}
\begin{gather}\label{2:asyK}
 K_\nu(x)=\sqrt{\frac\pi{2x}}e^{-x}\left(1+O\left(\frac1x\right)\right)\qquad\text{as} \quad 0<x\to\infty.
\end{gather}
This proves \eqref{2:beta}.
\end{proof}

\begin{Lemma}\label{2:l3}\quad
\begin{enumerate}\itemsep=0pt
\item[$(a)$] If $\Re \mu>0$ then, for every $N=1,2,3,\dots$, we have
\begin{gather}\label{2:limitW2}
\limsup_{z\to 0^+}\left| z^{\mu-1}W_2(u,z)-\Gamma(\mu)2^{\mu-1}u^{-\mu}\left(1-2\mu\sum_{s=0}^{N-1}\frac{B_s'(0)}{u^{2s+2}}\right)\right|=
O\left(\frac{u^{-\mu}}{u^{2N+2}}\right)
\end{gather}
as $0<u\to\infty$.
\item[$(b)$] If $\Re \mu=0$, $\mu\neq0$, \eqref{2:limitW2} holds when we replace $z^{\mu-1}W_2(u,z)$ by
\begin{gather*}
z^{\mu-1}W_2(u,z)-\Gamma(-\mu)2^{-\mu-1}u^\mu z^{2\mu}.
\end{gather*}
\item[$(c)$] If $\mu=0$ then
\begin{gather*}%\label{2:limitW2c}
\limsup_{z\to 0^+} \left|\frac{W_2(u,z)}{z\ln z}+1\right|=O\left(\frac{1}{u^{2N}}\right) .
\end{gather*}
\end{enumerate}
\end{Lemma}

\begin{proof}
Suppose that $\Re\mu>0$.
Then we use \cite[(10.30.2)]{NIST}
\begin{gather*}%\label{2:limitK}
\lim_{x\to 0^+} x^\nu K_\nu(x)=\Gamma(\nu)2^{\nu-1}\qquad\text{for} \quad \Re\nu>0.
\end{gather*}
It follows that
\begin{gather}
 \lim_{z\to 0^+} z^\mu K_\mu(uz) = \Gamma(\mu)2^{\mu-1}u^{-\mu}, \label{2:limitK1}\\
 \lim_{z\to 0^+} \frac1u z^{\mu+1} K_{\mu+1}(uz) = \Gamma(\mu+1)2^\mu u^{-\mu-2}.\label{2:limitK2}
\end{gather}
Using \eqref{1:W2}, \eqref{1:As}, \eqref{2:limitK1}, \eqref{2:limitK2}, we obtain
\begin{gather*}
 \limsup_{z\to0^+}\left|z^{\mu-1}W_2(u,z)-\Gamma(\mu)2^{\mu-1}u^{-\mu}\left(1-2\mu\sum_{s=0}^{N-1}\frac{B_s'(0)}{u^{2s+2}}\right)\right|\\
\qquad{}\le\limsup_{z\to0^+} \left|\Gamma(\mu)2^{\mu-1}u^{-\mu}g_2(u,z)- \Gamma(\mu+1)2^\mu u^{-\mu-2}h_2(u,z)\right| .
\end{gather*}
Now \eqref{1:W2a} gives \eqref{2:limitW2} with $N-1$ in place of $N$. If $\Re\mu=0$, $\mu\neq0$, then we use \cite[(9.7)]{O}
\begin{gather*}%\label{2:limitK3}
K_\mu(x)=\Gamma(\mu)2^{\mu-1}x^{-\mu}+\Gamma(-\mu)2^{-\mu-1}x^\mu+o(1)\qquad\text{as}\quad 0<x\to0
\end{gather*}
and argue similarly. If $\mu=0$ we use \cite[(10.30.3)]{NIST}
\begin{gather*}
\lim_{x\to 0^+} \frac{K_0(x)}{\ln x} =-1 .\tag*{\qed}
\end{gather*}
\renewcommand{\qed}{}
\end{proof}

\begin{Theorem}\label{2:t1}
Suppose that $\Re\mu\ge 0$ and $\mu$ is not an integer.
There are functions $\gamma(u)$, $\delta(u)$ such that
\begin{gather}\label{2:connectW2}
W_2(u,z)=\gamma(u)W_3(u,\mu,z)+\delta(u)W_3(u,-\mu,z) ,
\end{gather}
and, for every $N=1,2,3,\dots$,
\begin{gather}
 \gamma(u) = -\frac{\pi}{2\sin(\pi\mu)}\left(1+O\left(\frac{1}{u^{2N}}\right)\right),\label{2:gamma}\\
 \delta(u) = \frac{\pi}{2\sin(\pi\mu)}\left(1-2\mu\sum_{s=0}^{N-1}\frac{B_s'(0)}{u^{2s+2}}+O\left(\frac{1}{u^{2N+2}}\right)\right).\label{2:delta}
\end{gather}
\end{Theorem}
\begin{proof}
Since $\mu$ is not an integer, $W_3(u,\mu,z)$ and $W_3(u,-\mu,z)$ are linearly independent so \eqref{2:connectW2} holds for
some suitable functions~$\gamma$,~$\delta$.
From~\eqref{2:connectW2} we get
\begin{gather*}
W_2\big(u,ze^{\pi i}\big)-e^{\pi i (1-\mu)}W_2(u,z)=\gamma(u)\big(e^{\pi i(1+\mu)}-e^{\pi i(1-\mu)}\big) W_3(u,z) .
\end{gather*}
Comparing with Lemma~\ref{2:l2}, we f\/ind
$-2i\gamma(u)\sin(\pi\mu) =\beta(u)$.
Now~\eqref{2:beta} gives~\eqref{2:gamma}.

Suppose that $\Re\mu>0$. Then \eqref{2:connectW2} yields
\begin{gather*}
\lim_{z\to0^+} z^{\mu-1}W_2(u,z)=\delta(u)\frac{2^\mu u^{-\mu}}{\Gamma(1-\mu)}.
\end{gather*}
Using Lemma \ref{2:l3}(a) we obtain
\begin{gather*}
\Gamma(\mu)2^{\mu-1}u^{-\mu}\left(1-2\mu\sum_{s=0}^{N-1} \frac{B_s'(0)}{u^{2s+2}}+O\left(\frac{1}{u^{2N+2}}\right)\right)=\delta(u)\frac{2^\mu u^{-\mu}}{\Gamma(1-\mu)}.
\end{gather*}
Applying the ref\/lection formula for the Gamma function, we obtain~\eqref{2:delta}.
If $\Re\mu=0$, $\mu\neq0$, the proof of~\eqref{2:delta} is similar.
\end{proof}

\section[Removal of restriction on $\arg z$]{Removal of restriction on $\boldsymbol{\arg z}$}\label{removal}

Using $\beta(u)$ from Lemma \ref{2:l2} we def\/ine
\begin{gather*} W_4(u,z)=\frac{\pi i}{\beta(u)} W_2(u,z) .\end{gather*}
Then we have
\begin{gather}\label{3:connect1}
W_4\big(u,ze^{\pi i}\big)=e^{\pi i(1-\mu)} W_4(u,z)+\pi i W_3(u,z) .
\end{gather}
Moreover, \eqref{2:beta} shows that $W_4$ shares the asymptotic expansion \eqref{1:W2}, \eqref{1:W2a} with~$W_2$.
From~\eqref{3:connect1} we obtain
\begin{gather}\label{3:connect}
W_4\big(u,ze^{\pi im}\big)=e^{\pi i (1-\mu)m} W_4(u,z)+\pi i\frac{\sin(\pi(\mu+1)m)}{\sin(\pi(\mu+1))} W_3(u,z)
\end{gather}
for every integer $m$.
We will use~\eqref{3:connect} and the asymptotic expansions \eqref{1:W1}, \eqref{1:W2} for $|\arg z|\le \frac12\pi$
to prove that in~\eqref{1:W2a} we can remove the restriction on~$\arg z$ completely.

\begin{Theorem}\label{3:t1}
Suppose that $\Re\mu\ge0$. For every $N=1,2,3,\dots$, $W_2(u,z)$ can be written as the right-hand side of~\eqref{1:W2},
and \eqref{1:W2a} holds without a restriction on~$\arg z$:
\begin{gather*}%\label{3:W2b}
|g_2(u,z)|+|h_2(u,z)|\le \frac{K_2}{u^{2N}}\qquad\text{for} \quad 0<|z|\le R, \quad u\ge u_2.
\end{gather*}
\end{Theorem}

\begin{proof}
Without loss of generality we replace $W_2$ by $W_4$. We assume that $|\arg z|\le \frac12\pi$, $0<|z|\le R$, $u>0$, $m$ is an integer and $z_1:=ze^{\pi im}$.
We insert~\eqref{1:W1}, \eqref{1:W2} on the right-hand side of~\eqref{3:connect}.
Using~\eqref{3:ancontK} we obtain
\begin{gather}\label{3:eq4}
W_4(u,z_1)=z_1K_\mu(uz_1)\sum_{s=0}^{N-1} \frac{A_s(z_1)}{u^{2s}}
-\frac{z_1}{u}K_{\mu+1}(uz_1)\sum_{s=0}^{N-1} \frac{B_s(z_1)}{u^{2s}}+f(u,z),
\end{gather}
where
\begin{gather*}%\label{3:f}
 f=E_1g_2+E_2h_2+E_3g_1+E_4h_1,
\end{gather*}
with
\begin{alignat*}{3}
& E_1(u,z) = e^{-\pi i(\mu+1)m} zK_\mu(uz),\qquad &&
 E_2(u,z) = -e^{-\pi i (\mu+1)m}\frac{z^2}{u} K_{\mu+1}(uz),&\\
& E_3(u,z) = \pi i\frac{\sin(\pi(\mu+1)m)}{\sin(\pi(\mu+1))}z I_\mu(uz),\qquad&&
 E_4(u,z) = \pi i \frac{\sin(\pi (\mu+1)m)}{\sin(\pi (\mu+1))}\frac{z^2}{u}I_{\mu+1}(uz) .&
\end{alignat*}
We will construct functions $G_j(u,z)$ and $H_j(u,z)$ such that
\begin{gather*} E_j(u,z)=z_1 K_\mu(uz_1) G_j(u,z)-\frac{z_1^2}{u} K_{\mu+1}(uz_1) H_j(u,z)
\end{gather*}
for $j=1,2,3,4$.
Then \eqref{3:eq4} becomes
\begin{gather}
W_4(u,z_1)=z_1K_\mu(uz_1)\left(\sum_{s=0}^{N-1} \frac{A_s(z_1)}{u^{2s}}+g_3(u,z)\right)\nonumber\\
\hphantom{W_4(u,z_1)=}{}
-\frac{z_1}{u}K_{\mu+1}(uz_1)\left(\sum_{s=0}^{N-1} \frac{B_s(z_1)}{u^{2s}}+z_1h_3(u,z)\right),\label{4:eq5}
\end{gather}
where
\begin{gather*}
g_3 = G_1 g_2+G_2h_2+G_3 g_1+G_4 h_1,\qquad
h_3 = H_1 g_2+H_2h_2+H_3 g_1+H_4 h_1.
\end{gather*}
We now use \cite[(10.28.2)]{NIST}
\begin{gather}\label{3:relation}
 K_\mu(x)I_{\mu+1}(x)+K_{\mu+1}(x)I_\mu(x)=\frac1x .
\end{gather}
From \eqref{3:relation} and the relation
\begin{gather}\label{3:ancontI}
 I_\mu\big(ze^{\pi im}\big)=e^{\pi i\mu m}I_\mu(z)
\end{gather}
we obtain
\begin{gather*} u z_1K_\mu(uz_1) e^{\pi i (\mu+1)m} I_{\mu+1}(uz)+uz_1K_{\mu+1}(uz_1) e^{\pi i \mu m} I_\mu(uz)=1 .\end{gather*}
Therefore, we can choose
\begin{gather*}
G_1(u,z) = uz K_\mu(uz)I_{\mu+1}(uz),\qquad
H_1(u,z) = -u^2K_\mu(uz)I_\mu(uz) .
\end{gather*}
We set
\begin{gather*}
l_0(x)=\ln\frac{1+2|x|}{|x|},\qquad l_\mu(x)=1\qquad\text{if} \quad \mu\ne 0,
\end{gather*}
and note the estimates \cite[(9.12)]{O}
\begin{alignat}{3}
&|I_\mu(x)K_\mu(x)|\le \frac{C l_\mu(x)}{1+|x|},\qquad&& |I_{\mu+1}(x)K_\mu(x)|\le \frac{C |x|l_\mu(x)}{1+|x|^2},& \label{3:olver1}\\
& |I_{\mu+1}(x)K_{\mu+1}(x)|\le \frac{C}{1+|x|},\qquad&& |I_\mu(x)K_{\mu+1}(x)|\le \frac{C}{|x|}& \label{3:olver2}
\end{alignat}
valid when $|\arg x|\le \frac12\pi$ with $C$ independent of $x$.
At this point we assume that $\mu\ne0$ (the case $\mu=0$ is mentioned at the end of the proof). The estimates~\eqref{3:olver1} give
\begin{gather}\label{3:est1}
|G_1(u,z)|\le C,\qquad |H_1(u,z)|\le C u^2.
\end{gather}
Similarly, we choose
\begin{gather*}
G_2(u,z) = -z^2K_{\mu+1}(uz)I_{\mu+1}(uz),\qquad
H_2(u,z) = uzK_{\mu+1}(uz)I_\mu(uz).
\end{gather*}
The estimates \eqref{3:olver2} give
\begin{gather}\label{3:est2}
|G_2(u,z)|\le C|z|^2,\qquad|H_2(u,z)|\le C.
\end{gather}
It follows from \eqref{3:ancontK} that
\begin{gather*}
E_3(u,z) = -E_1(u,z)+ z_1K_\mu(uz_1),\qquad
E_4(u,z) = -E_2(u,z)-\frac{z_1^2}{u}K_{\mu+1}(uz_1) .
\end{gather*}
Therefore, we can choose
\begin{alignat*}{3}
& G_3(u,z) = 1-G_1(u,z) ,\qquad &&
H_3(u,z) = -H_1(u,z),&\\
& G_4(u,z) = -G_2(u,z),\qquad&&
H_4(u,z) = 1-H_2(u,z).&
\end{alignat*}
From \eqref{3:est1}, \eqref{3:est2}, we get
\begin{alignat}{3}
& |G_3(u,z)|\le C+1,\qquad && |H_3(u,z)|\le C u^2,& \label{3:est3}\\
& |G_4(u,z)|\le C|z|^2,\qquad && |H_4(u,z)|\le C+1 .&\label{3:est4}
\end{alignat}
The estimates \eqref{3:est1}, \eqref{3:est2}, \eqref{3:est3}, \eqref{3:est4}
give
\begin{gather*}
|g_3(u,z)| \le C |g_2(u,z)| + C|z|^2|h_2(u,z)|+(C+1)|g_1(u,z)|+C|z|^2 |h_1(u,z)|,\\
|h_3(u,z)| \le C u^2|g_2(u,z)| + C|h_2(u,z)|+Cu^2|g_1(u,z)|+(C+1) |h_1(u,z)|.
\end{gather*}
Since we assumed that
\begin{gather*}
|g_1(u,z)|+|h_1(u,z)|+|g_2(u,z)|+|h_2(u,z)|\le \frac{K}{u^{2N}}
\end{gather*}
for $| \arg z|\le \frac12\pi$, $0<|z|\le R$, $u\ge u_0$,
the expansion \eqref{4:eq5} has the desired form with $N$ replaced by $N-1$.

Suppose $\mu=0$.
We use \cite[(10.31.2)]{NIST}
\begin{gather}\label{3:K0}
K_0(x)=-\left(\ln\left(\frac12x\right)+\gamma\right)I_0(x)+\frac{\frac14x^2}{(1!)^2}+\left(1+\frac12\right)
\frac{\left(\frac14x^2\right)^2}{\left(2!\right)^2}+\cdots.
\end{gather}
It follows from \eqref{3:K0} that there exist positive constants $r>0$, $D>0$ such that
\begin{gather*}%\label{3:est5}
\frac{|K_0(x)|}{|K_0(xe^{\pi i m})|}\le D \qquad\text{for} \quad 0<|x|\le r, \quad |\arg x|\le\frac12\pi, \quad m\in\mathbb{Z}.
\end{gather*}
Then we set
\begin{gather*}
G_1(u,z)=\frac{K_0(uz)}{K_0(uz_1)},\qquad H_1(u,z)=0\qquad\text{if} \quad 0<|uz|\le r
\end{gather*}
with $G_1$ and $H_1$ the same as before when $|uz|>r$. The estimates~\eqref{3:est1} are valid with a suitable constant~$C$.
The rest of the proof is unchanged.
This completes the proof of the theorem.
\end{proof}

\section[Extension to complex $u$]{Extension to complex $\boldsymbol{u}$}\label{extension}

So far we considered only $0<u\to\infty$. Now we set $u=te^{i\theta}$, where $t>0$ and $\theta\in\R$.
In~\eqref{1:eq1} we substitute $z=e^{-i\theta} x$, $\tilde w(x)=w(z)$.
Then we obtain the dif\/ferential equation
\begin{gather}\label{4:eq1}
\frac{d^2}{dx^2} \tilde w(x)=\frac{1}{x}\frac{d}{dx} \tilde w(x) +\left(t^2 +\frac{\mu^2-1}{x^2}+e^{-2i\theta}f\big(e^{-i\theta} x\big)\right)\tilde w(x) .
\end{gather}
Assuming $\Re\mu\ge0$, we can apply Olver's theory to this equation, and obtain functions $\tilde W_1(t,x)$ and $\tilde W_2(t,x)$.
Since we assumed that $f(z)$ is analytic in the disk $\{z\colon |z|<R_0\}$, the new function
$\tilde f(x)= e^{-2i\theta}f(e^{-i\theta} x)$ is analytic in the same disk. Therefore, the domains $D_1$, $D_2$ are the same as before.
The functions~$\tilde A_s(x)$, $\tilde B_s(x)$ that appear in place of~$A_s(z)$,~$B_s(z)$ satisfy
\begin{gather*}
\tilde A_s(x)=e^{-2si\theta} A_s(z),\qquad \tilde B_s(x)=e^{-(2s+1)i\theta} B_s(z) ,
\end{gather*}
so
\begin{gather*}
\frac{\tilde A_s(x)}{t^{2s}} =\frac{A_s(z)}{u^{2s}},\qquad \frac{\tilde B_s(x)}{t^{2s+1}} =\frac{B_s(z)}{u^{2s+1}}.
\end{gather*}
Therefore, the functions $e^{{-}i\theta} \tilde W_1(t{,}x)$ and $e^{{-}i\theta} \tilde W_2(t{,}x)$ have the asymptotic expan\-sions~\eqref{1:W1},~\eqref{1:W1a} and~\eqref{1:W2},~\eqref{1:W2a} with $(t,x)$ replacing $(u,z)$.

Let $\tilde W_3(t,\mu,x)$ be the function $W_3$ for the dif\/ferential equation~\eqref{4:eq1}.
Then
\begin{gather*}
W_3\big(te^{i\theta},\mu,e^{-i\theta}x\big)=e^{-i\theta} \tilde W_3(t,\mu,x).
\end{gather*}
It follows that $W_3(u,\mu,z)$ can be expanded in the form of the right-hand side of~\eqref{1:W1}, and~\eqref{1:W1a} holds for $0<|z|\le R$
and $u=te^{i\theta}$ for any f\/ixed real~$\theta$.

We would like to connect $\tilde W_2$ to $W_2$ in a similar manner but this is not possible at this point because $W_2(u,z)$ is only def\/ined for
$u>0$, and so we cannot substitute $u=te^{i\theta}$.

\section[Properties of $A_s$, $B_s$]{Properties of $\boldsymbol{A_s}$, $\boldsymbol{B_s}$}\label{AB}

For any $\mu\in\C$ we consider the solution $A_s(z)=A_s(\mu,z)$, $B_s(z)= B_s(\mu,z)$ of the recur\-sion~\eqref{1:eq3},~\eqref{1:eq4}
which is uniquely determined by $A_0(z)=1$ and~\eqref{1:As}.
The following lemma is mentioned by
Olver \cite[p.~327]{O1}, \cite[p.~81, line~6]{O}.

\begin{Lemma}\label{5:l1}
Let $\hat A_s(z)$, $\hat B_s(z)$ be any solution of \eqref{1:eq3}, \eqref{1:eq4} with $\hat A_0(z)=1$. Then, for all $s\ge 0$,
\begin{gather}\label{5:eq1}
\hat A_s(z)=\sum_{r=0}^s A_r(z)\hat A_{s-r}(0),\qquad \hat B_s(z)=\sum_{r=0}^s B_r(z)\hat A_{s-r}(0) .
\end{gather}
\end{Lemma}
\begin{proof}
Let us denote the right-hand sides of equations \eqref{5:eq1} by $A_s^\ast(z)$, $B_s^\ast(z)$, respectively.
It is easy to show that $A_s^\ast(z)$, $B_s^\ast(z)$ is a solution of~\eqref{1:eq3},~\eqref{1:eq4}. Since $A_0^\ast(z)=1$ and $A_s^\ast(0)=\hat A_s(0)$, this solution must
agree with~$\hat A_s(z)$,~$\hat B_s(z)$.
\end{proof}

We now def\/ine $a_0(z)=1$ and, for $s\ge 0$,
\begin{gather}
 a_{s+1}(z) := A_{s+1}(-\mu,z)+\frac{2\mu}{z} B_s(-\mu,z),\label{5:a}\\
 b_s(z) := B_s(-\mu,z) .\label{5:b}
\end{gather}

\begin{Theorem}\label{5:t1}
The functions $a_s(z)$, $b_s(z)$ satisfy \eqref{1:eq3}, \eqref{1:eq4} with $A_s$, $B_s$ replaced by $a_s$, $b_s$, respectively, and, for all $s\ge 0$,
\begin{gather}
a_s(z) = A_s(\mu,z)+2\mu\sum_{r=0}^{s-1} A_r(\mu,z)B_{s-1-r}'(-\mu,0),\label{5:c1}\\
b_s(z) = B_s(\mu,z)+2\mu\sum_{r=0}^{s-1} B_r(\mu,z)B_{s-1-r}'(-\mu,0).\label{5:c2}
\end{gather}
\end{Theorem}
\begin{proof}
We have
\begin{gather*} 2 A_{s+1}(-\mu,z)=\frac{-2\mu+1}{z} B_s(-\mu,z)- B_s'(-\mu,z)+\int f(z) B_s(-\mu,z) dz .\end{gather*}
We add $\frac{4\mu}{z} B_s(-\mu,z)$ on both sides and get
\begin{gather}\label{5:eq2}
2a_{s+1}(z)=\frac{2\mu+1}{z}b_s(z)-b_s'(z)+\int f(z)b_s(z) dz .
\end{gather}
This is \eqref{1:eq4} for $a_s(z)$, $b_s(z)$.

Equation \eqref{1:eq3} is true for $a_s(z)$, $b_s(z)$ when $s=0$. Suppose $s\ge 1$.
We have
\begin{gather*} 2 B_s'(-\mu,z)=- A_s''(-\mu,z)+f(z) A_s(-\mu,z)+\frac{2\mu-1}{z} A_s'(-\mu,z) .\end{gather*}
Using the def\/initions of $a_s(z)$, $b_s(z)$ we get
\begin{gather}\label{5:eq3}
 2b_s'(z)=-a_s''(z)+f(z)a_s(z)-\frac{2\mu+1}{z}a_s'(z)+\frac{4\mu}{z}a_s'(z)+G,
\end{gather}
where
\begin{gather*} G:=\frac{d^2}{dz^2}\left(\frac{2\mu}{z} b_{s-1}(z)\right)-f(z)\frac{2\mu}{z}b_{s-1}(z)
-\frac{2\mu-1}{z}\frac{d}{dz}\left(\frac{2\mu}{z} b_{s-1}(z)\right) .
\end{gather*}
In \eqref{5:eq3} we replace $\frac{4\mu}{z}a_s'(z)$ through~\eqref{5:eq2}. Then we obtain
\begin{gather}\label{5:eq4}
2b_s'(z)=-a_s''(z)+f(z)a_s(z)-\frac{2\mu+1}{z}a_s'(z)+H+G,
\end{gather}
where
\begin{gather*}
H:=\frac{2\mu}{z}\left[\frac{d}{dz}\left(\frac{2\mu+1}{z}b_{s-1}(z)\right)-b_{s-1}''(z)+f(z)b_{s-1}(z)\right] .
\end{gather*}
By direct computation, we show $H+G=0$ for any function $b_{s-1}(z)$. Therefore, by integra\-ting~\eqref{5:eq4}
noting that $a_s(z)$ is even and $b_s(z)$ is odd, we obtain \eqref{1:eq3} for $a_s(z)$, $b_s(z)$.

We now get \eqref{5:c1}, \eqref{5:c2} from Lemma \ref{5:l1}.
\end{proof}

Using multiplication of formal series, we can write \eqref{5:c1}, \eqref{5:c2} as
\begin{gather}
F(u,-\mu)\sum_{s=0}^\infty \frac{A_s(z)}{u^{2s}} = \sum_{s=0}^\infty \frac{a_s(z)}{u^{2s}},\label{5:formal1}\\
F(u,-\mu)\sum_{s=0}^\infty \frac{B_s(z)}{u^{2s}} = \sum_{s=0}^\infty \frac{b_s(z)}{u^{2s}},\label{5:formal2}
\end{gather}
where
\begin{gather*}%\label{5:F}
F(u,\mu)=1-2\mu\sum_{s=0}^\infty \frac{B_s'(\mu,0)}{u^{2s+2}}.
\end{gather*}
We dif\/ferentiate \eqref{5:c2} with respect to $z$ and set $z=0$.
Then we f\/ind
\begin{gather*}%\label{5:c3}
B_s'(-\mu,0)=B_s'(\mu,0)+2\mu\sum_{r=0}^{s-1} B_r'(\mu,0)B_{s-1-r}'(-\mu,0),
\end{gather*}
or, equivalently,
\begin{gather}\label{5:reciprocal}
F(u,\mu)F(u,-\mu)=1 .
\end{gather}
In particular, it follows that
\begin{gather}
F(u,\mu)\sum_{s=0}^\infty \frac{a_s(z)}{u^{2s}}=\sum_{s=0}^\infty \frac{A_s(z)}{u^{2s}},\label{5:formal3}\\
F(u,\mu)\sum_{s=0}^\infty \frac{b_s(z)}{u^{2s}}=\sum_{s=0}^\infty \frac{B_s(z)}{u^{2s}}.\label{5:formal4}
\end{gather}

\section[Asymptotic expansion of $W_3$ when $\Re\mu<0$]{Asymptotic expansion of $\boldsymbol{W_3}$ when $\boldsymbol{\Re\mu<0}$} \label{W3}

In Section \ref{discussion} we saw that $W_3(u,\mu,z)$ can be written as the right-hand side of~\eqref{1:W1}, and~\eqref{1:W1a} holds.
However, this was proved only when $\Re\mu\ge 0$. Now we remove this restriction.

\begin{Theorem}\label{6:t1}
Suppose that $\mu\in\C$ is not a negative integer, and $u=te^{i\theta}$ with $t>0$, $\theta\in\R$.
Then $W_3(u,\mu,z)$ can be written as the right-hand side of~\eqref{1:W1} and, for each $R>0$ and $N\ge 1$, there are constants $L_1$ and $t_1$ such that
\begin{gather*}%\label{6:W3a}
|g_1(u,z)|+|h_1(u,z)|\le \frac{L_1}{t^{2N}}\qquad \text{for} \quad 0<|z|\le R, \quad t\ge t_1.
\end{gather*}
\end{Theorem}

\begin{proof}
In Sections \ref{discussion} and \ref{extension} we proved this statement for $\Re\mu\ge 0$. Therefore, it will be suf\/f\/icient to treat $W_3(u,-\mu,z)$ with $\Re\mu>0$.
By the considerations in Section~\ref{extension}, it is suf\/f\/icient to consider $\theta=0$, so $u>0$.
Suppose $|\arg z|\le \frac12\pi$, $0<|z|\le R$. By~\eqref{2:connectW2}, we have
\begin{gather}\label{6:eq1}
c \delta(u)W_3(u,-\mu,z)=c W_2(u,\mu,z)-c\gamma(u)W_3(u,\mu,z),
\end{gather}
where $c=\frac{2}{\pi}\sin(\pi\mu)$. On the right-hand side of \eqref{6:eq1} we insert the expansions \eqref{1:W1} for~$W_3$ and~\eqref{1:W2} for~$W_2$. Taking into account \eqref{2:gamma}, we can expand $-c\gamma(u)W_3(u,\mu,z)$ the same way as~$W_3$. Then using \cite[(10.27.4)]{NIST}
\begin{gather}\label{6:eq2}
K_\nu(x)=\frac{\pi}{2\sin(\pi\nu)}\left(I_{-\nu}(x)-I_\nu(x)\right),
\end{gather}
we obtain
\begin{gather}\label{6:eq3}
c\delta(u)W_3(u,-\mu,z)=zI_{-\mu}(uz)\sum_{s=0}^{N-1}\frac{A_s(z)}{u^{2s}}+\frac{z}{u}I_{-\mu-1}(uz)\sum_{s=0}^{N-1} \frac{B_s(z)}{u^{2s}}+f(u,z),
\end{gather}
where
\begin{gather*}%\label{6:eq4}
 f=E_1g_2+E_2h_2+E_3g_1+E_4h_1
\end{gather*}
with
\begin{alignat*}{3}
& E_1(u,z) = c zK_\mu(uz),\qquad && E_2(u,z) = -c\frac{z^2}{u} K_{\mu+1}(uz),& \\
& E_3(u,z) = z I_\mu(uz),\qquad && E_4(u,z) = \frac{z^2}{u}I_{\mu+1}(uz) . &
\end{alignat*}
We will construct functions $G_j(u,z)$ and $H_j(u,z)$ such that
\begin{gather*}
E_j(u,z)=z I_{-\mu}(uz) G_j(u,z)+\frac{z^2}{u} I_{1-\mu}(uz) H_j(u,z)
\end{gather*}
for $j=1,2,3,4$.
Also using \cite[(10.29.1)]{NIST}
\begin{gather}\label{6:eq5}
I_{\nu-1}(x)-I_{\nu+1}(x)=\frac{2\nu}{x}I_\nu(x),
\end{gather}
\eqref{6:eq3} becomes
\begin{gather}
c\delta(u)W_3(u,-\mu,z) = zI_{-\mu}(uz)\left(\sum_{s=0}^{N-1} \frac{\tilde A_s(z)}{u^{2s}}+g_3(u,z)\right)\nonumber\\
\hphantom{c\delta(u)W_3(u,-\mu,z) =}{} +\frac{z}{u}I_{1-\mu}(uz)\left(\sum_{s=0}^{N-1} \frac{B_s(z)}{u^{2s}}+zh_3(u,z)\right),\label{6:eq6}
\end{gather}
where
\begin{gather*} \tilde A_0(z)=1,\qquad \tilde A_s(z)=A_s(z)-\frac{2\mu}{z}B_{s-1}(z) \qquad \text{for} \quad s=1,\dots,N-1,
\end{gather*}
and
\begin{gather*}
g_3 = -\frac{2\mu}{z}B_{N-1}(z)u^{-2N}+ G_1 g_2+G_2h_2+G_3 g_1+G_4 h_1,\\
h_3 = H_1 g_2+H_2h_2+H_3 g_1+H_4 h_1.
\end{gather*}
The identities \eqref{3:relation} and \cite[(10.29.1)]{NIST}
\begin{gather}\label{6:eq7}
K_{\nu-1}(x)-K_{\nu+1}(x)=-\frac{2\nu}{x}K_\nu(x)
\end{gather}
give
\begin{gather*}%\label{6:eq8}
uzI_{-\mu}(uz)\left(K_{\mu+1}(uz)-\frac{2\mu}{uz}K_\mu(uz)\right)+uzI_{1-\mu}(uz)K_\mu(uz)=1 .
\end{gather*}
Therefore, we can choose
\begin{gather*}
G_3(u,z) = uz\left(K_{\mu+1}(uz)-\frac{2\mu}{uz}K_\mu(uz)\right)I_\mu(uz),\\
H_3(u,z) = u^2K_\mu(uz)I_\mu(uz) .
\end{gather*}
The estimates \eqref{3:olver1}, \eqref{3:olver2} give
\begin{gather}\label{6:est1}
|G_3(u,z)|\le C_3,\qquad |H_3(u,z)|\le D_3 u^2.
\end{gather}
Similarly, we choose
\begin{gather*}
G_4(u,z) = z^2\left(K_{\mu+1}(uz)-\frac{2\mu}{uz}K_\mu(z)\right)I_{\mu+1}(uz),\\
H_4(u,z) = uzK_\mu(uz)I_{\mu+1}(uz),
\end{gather*}
and estimate
\begin{gather}\label{6:est2}
|G_4(u,z)|\le C_4|z|^2,\qquad|H_4(u,z)|\le D_4.
\end{gather}
It follows from \eqref{6:eq2}, \eqref{6:eq5} that
\begin{gather*}
E_1(u,z) = zI_{-\mu}(uz)-E_3(u,z),\\
E_2(u,z) = \frac{z^2}{u}\left(-\frac{2\mu}{uz}I_{-\mu}(uz)+I_{-\mu+1}(uz)\right)-E_4(u,z) .
\end{gather*}
Therefore, we can choose
\begin{alignat*}{3}
& G_1(u,z) = 1-G_3(u,z) ,\qquad && H_1(u,z) = -H_3(u,z),& \\
& G_2(u,z) = -\frac{2\mu}{u^2}-G_4(u,z),\qquad && H_2(u,z) = 1-H_4(u,z).&
\end{alignat*}
From \eqref{6:est1}, \eqref{6:est2}, we get
\begin{alignat}{3}
& |G_1(u,z)|\le C_1,\qquad && |H_1(u,z)|\le D_1 u^2,& \label{6:est3}\\
&|G_2(u,z)|\le C_2(1+|z|^2),\qquad && |H_2(u,z)|\le D_2 .&\label{6:est4}
\end{alignat}
Since we know that
\begin{gather*}
|g_1(u,z)|+|h_1(u,z)|+|g_2(u,z)|+|h_2(u,z)|\le \frac{K}{u^{2N}}
\end{gather*}
for $| \arg z|\le \frac12\pi$, $0<|z|\le R$, $u\ge u_0$,
the estimates \eqref{6:est1}, \eqref{6:est2}, \eqref{6:est3}, \eqref{6:est4}
give
\begin{gather*}%\label{6:est5}
|g_3(u,z)|+|h_3(u,z)|\le \frac{L}{u^{2N-2}}\qquad\text{if} \quad |\arg z|\le \frac12\pi, \quad 0<|z|\le R, \quad u\ge u_3.
\end{gather*}
Now we divide both sides of \eqref{6:eq6} by $c\delta(u)$ and use \eqref{2:delta}, \eqref{5:formal3}, \eqref{5:formal4} (with $\mu$ replaced by $-\mu$).
Then we obtain the desired expansion of $W_3(u,-\mu,z)$ for $\Re\mu<0$ and $|\arg z|\le \frac12\pi$, $0<|z|\le R$.
The restriction on $\arg z$ is easily removed using~\eqref{3:ancontI} and $W_3(e^{\pi im}z)=e^{\pi i (\mu+1)m} W_3(z)$.
\end{proof}

\section{Application to the conf\/luent hypergeometric equation}\label{confluent}
The conf\/luent hypergeometric dif\/ferential equation
\begin{gather*}%\label{M:eq1}
xv''(x)+(b-x)v'(x)-av(x)=0
\end{gather*}
has solutions $M(a,b,x)$ and $U(a,b,x)$.
Substituting $x=z^2$, $w=e^{-\frac12 z^2}z^b v$ we obtain the dif\/ferential equation
\begin{gather}\label{M:eq2}
w''(z)=\frac1z w'(z)+\left(u^2+\frac{\mu^2-1}{z^2}+z^2\right)w(z),
\end{gather}
where
\begin{gather}\label{M:a}
 a=\frac14 u^2+\frac12 b, \qquad \mu=b-1.
\end{gather}
Equation \eqref{M:eq2} agrees with~\eqref{1:eq1} when $f(z)=z^2$.
Let $A_s$, $B_s$ be def\/ined as in Section~\ref{Olver} for $f(z)=z^2$. In this case, $A_s(z)$, $B_s(z)$ are polynomials.
Throughout this section, we assume that $a$, $b$, $u$, $\mu$ satisfy \eqref{M:a}.

The function $M(a,b,x)$ is given by a power series in $x$ and $M(a,b,0)=1$. Therefore, the function~$W_3$ associated with~\eqref{M:eq2}
is given by
\begin{gather}\label{M:W3}
W_3(u,\mu,z)=\frac{2^{1-b}u^{b-1}}{\Gamma(b)} e^{-\frac12z^2}z^b M\big(a,b,z^2\big).
\end{gather}
Theorem \ref{6:t1} implies the following theorem.

\begin{Theorem}\label{M:t1}
Suppose that $b\in\C$ is not $0$ or a negative integer, $u=te^{i\theta}$ with $t>0$, $\theta\in\R$, and $N\ge 1$, $R>0$.
Then we can write
\begin{gather}
 \frac{2^{1-b}u^{b-1}}{\Gamma(b)} e^{-\frac12z^2}z^b M\big(\tfrac14 u^2+\tfrac12 b,b,z^2\big)\nonumber\\
\qquad {} = zI_{b-1}(uz)\left(\sum_{s=0}^{N-1} \frac{A_s(z)}{u^{2s}}+g_1(u,z)\right)
 +\frac{z}{u}I_b(uz)\left(\sum_{s=0}^{N-1} \frac{B_s(z)}{u^{2s}}+zh_1(u,z)\right),\label{M:asy1}
\end{gather}
where
\begin{gather*}%\label{M:asy2}
|g_1(u,z)|+|h_1(u,z)|\le \frac{L_1}{t^{2N}}\qquad \text{for} \quad 0<|z|\le R, \quad t\ge t_1.
\end{gather*}
and $L_1$, $t_1$ are positive constants independent of $z$ and $u$ $($but possibly depending on~$b$, $\theta$, $N$, $R)$. There is no restriction on~$\arg z$. The polynomials $A_s(z)$, $B_s(z)$ appearing in~\eqref{M:asy1}
are determined by the recursion~\eqref{1:eq3},~\eqref{1:eq4} with $f(z)=z^2$ and the conditions $A_0(z)=1$, $A_s(0)=0$ for $s\ge 1$.
\end{Theorem}

Suppose that $\Re b\ge 1$. Let $W_2(u,z)$ be the function associated with equation~\eqref{M:eq2} which satisf\/ies~\eqref{1:W2},~\eqref{1:W2a}. There are functions $\beta_1(u)$, $\beta_2(u)$ such that
\begin{gather}\label{U:connect}
 W_2(u,z)=\beta_1(u) e^{-\frac12z^2}z^b M\big(a,b,z^2\big)+\beta_2(u) e^{-\frac12z^2}z^b U\big(a,b,z^2\big) .
\end{gather}

The determination of $\beta_1(u)$, $\beta_2(u)$ is not obvious.
It is in this part of the analysis where there is an error in \cite{S}. Slater \cite[p.~79]{S} derives $\beta_2(u)\sim\Gamma(a)2^{b-2}u^{1-b}$, and claims ``we can take $\beta_1(u)=0$'' without proof.
When comparing with~\cite{S}, note that our $\beta_2(u)$ is denoted by~$1/\beta_2(u)$ in~\cite{S}.
Actually, the stated formula for $\beta_2(u)$ is correct but it is only the leading term of the required full asymptotic expansion given in the following lemma.

\begin{Lemma}\label{U:l1}
Suppose $\Re b\ge 1$. For every $N=1,2,3,\dots$, as $0<u\to\infty$,
\begin{gather}\label{U:beta2}
\beta_2(u)=\Gamma(a)2^{b-2}u^{1-b}\left(1+2(1-b)\sum_{s=0}^{N-1}\frac{B_s'(0)}{u^{2s+2}}+O\left(\frac{1}{u^{2N+2}}\right)\right).
\end{gather}
\end{Lemma}

\begin{proof}
Suppose $\Re b>1$.
Then \cite[(13.2.18)]{NIST}
\begin{gather*}%\label{U:limitU1}
\lim_{z\to0^+}z^{2b-2}U\big(a,b,z^2\big)=\frac{\Gamma(b-1)}{\Gamma(a)}
\end{gather*}
and \eqref{U:connect} give
\begin{gather*}%\label{U:limit1}
\lim_{z\to 0^+} z^{b-2} W_2(u,z)= \beta_2(u) \frac{\Gamma(b-1)}{\Gamma(a)} .
\end{gather*}
Comparing with \eqref{2:limitW2}, we obtain~\eqref{U:beta2}.

If $\Re b=1$, $b\ne 1$, the proof is similar using Lemma~\ref{2:l3}(b) and~\cite[(13.2.18)]{NIST}
\begin{gather*}%\label{U:limitU2}
U(a,b,x)=\frac{\Gamma(b-1)}{\Gamma(a)} x^{1-b}+\frac{\Gamma(1-b)}{\Gamma(a-b+1)}+O(x)\qquad\text{as} \quad x\to0^+.
\end{gather*}
If $b=1$ we use Lemma \ref{2:l3}(c) and \cite[(13.2.19)]{NIST}
\begin{gather*}%\label{U:limitU3}
\lim_{x\to0^+} \frac{U(a,1,x)}{\ln x} =-\frac{1}{\Gamma(a)} .\tag*{\qed}
\end{gather*}
\renewcommand{\qed}{}
\end{proof}

We cannot show that $\beta_1(u)=0$ but we can prove that $|\beta_1(u)|$ is very small as $u\to\infty$.
To this end we need the following lemma.

\begin{Lemma}\label{U:l2}
Let $b\in\C$, $\Re x>0$, and $\epsilon>0$. There is a constant $Q$ independent of $a$ such that
\begin{gather*} |\Gamma(a) U(a,b,x)|\le Q \qquad\text{if} \quad \Re a\ge \epsilon. \end{gather*}
\end{Lemma}

\begin{proof}
We use the integral representation \cite[(13.4.4)]{NIST}
\begin{gather*}
\Gamma(a)U(a,b,x)=\int_0^\infty e^{-xt}t^{a-1}(1+t)^{b-a-1} dt .
\end{gather*}
Therefore, if $\Re a\ge \epsilon$,
\begin{gather*}
 |\Gamma(a)U(a,b,x)| \le \int_0^\infty e^{-\Re x t}\left(\frac{t}{1+t}\right)^{\Re a-\epsilon}\left(\frac{t}{1+t}\right)^{\epsilon-1}(1+t)^{\Re b-2} dt\\
\hphantom{|\Gamma(a)U(a,b,x)|}{} \le \int_0^\infty e^{-\Re x t}\left(\frac{t}{1+t}\right)^{\epsilon-1}(1+t)^{\Re b-2} dt = :Q .\tag*{\qed}
 \end{gather*}
\renewcommand{\qed}{}
\end{proof}

\begin{Lemma}\label{U:l3}
Suppose $\Re b\ge 1$. For every $q<R$ we have $\beta_1(u)=O(e^{-qu})$ as $0<u\to\infty$.
\end{Lemma}

\begin{proof}
In the following let $0<z\le R$ (and $b$) be f\/ixed.
By Lemmas~\ref{U:l1},~\ref{U:l2}, there is a constant $C_1>0$ such that, for suf\/f\/iciently large $u>0$,
\begin{gather}\label{U:est1}
 \big|\beta_2(u) e^{-\frac12z^2}z^bU\big(a,b,z^2\big)\big| \le C_1\big|u^{1-b}\big| Q .
\end{gather}
Using \eqref{2:asyI}
we get from Theorem \ref{M:t1} with $N=1$, for some constant $C_2>0$,
\begin{gather}\label{U:est2}
 \big|e^{-\frac12z^2} z^bM\big(a,b,z^2\big)\big|\ge C_2 \big|u^{\frac12-b}\big|e^{zu} .
\end{gather}
Similarly, \eqref{2:asyK},
\eqref{1:W2}, \eqref{1:W2a} yield
\begin{gather}\label{U:est3}
 |W_2(u,z)|\le C_3 u^{-\frac12}e^{-zu} .
\end{gather}
Substituting \eqref{U:est1}, \eqref{U:est2} and \eqref{U:est3} in~\eqref{U:connect}.
we f\/ind
\begin{gather*}
|\beta_1(u)|\le \frac{C_3}{C_2} \big|u^{b-1}\big|e^{-2zu}+\frac{C_1Q}{C_2}u^{\frac12}e^{-zu}.
\end{gather*}
If we choose $z=R$, we obtain the desired estimate.
\end{proof}

\begin{Lemma}\label{U:l4}
Suppose $\Re b\ge 1$. For every $N=1,2,3,\dots$, the function
\begin{gather*}%\label{U:Uasy}
 \beta_2(u)e^{-\frac12z^2}z^b U\big(a,b,z^2\big)
\end{gather*}
can be written in the form of the right-hand side of \eqref{1:W2}, and \eqref{1:W2a} holds with $R$ replaced by~$\frac13 R$.
\end{Lemma}

\begin{proof}
Let
\begin{gather*} L(u,z):=\beta_1(u)e^{-\frac12z^2}z^bM\big(a,b,z^2\big) .\end{gather*}
Applying Theorem~\ref{M:t1} and Lemma~\ref{U:l3}, we estimate, for $0<|z|\le R$,
\begin{gather}\label{U:est4}
 |L(u,z)|\le C_1 e^{-qu} |z|\big(|I_\mu(uz)|+u^{-1}|I_{\mu+1}(uz)|\big) ,
 \end{gather}
where $q<R$ will be chosen later.
We use the estimate
\begin{gather}\label{U:est5}
 |I_\nu(x)|\le C_2 e^{|x|} \qquad \text{for} \quad |\arg x|\le \frac32\pi
\end{gather}
provided that $\Re \nu\ge 0$. This inequality follows from \cite[(9.2), (9.3)]{O}.
Therefore, \eqref{U:est4} yields
\begin{gather}\label{U:est6}
|L(u,z)|\le C_3 |z|e^{-qu+\frac13Ru}\qquad\text{for}\quad 0<|z|\le \frac13R, \quad |\arg z|\le \frac32\pi, \quad u\ge u_0.
\end{gather}
Using \eqref{3:relation}, we have
\begin{gather*} L(u,z)=zK_\mu(uz)g(u,z)-\frac{z}{u}K_{\mu+1}(uz) zh(u,z), \end{gather*}
where
\begin{gather*}
g(u,z) = uI_{\mu+1}(uz) L(u,z),\qquad
h(u,z) = -\frac{u^2}{z} I_\mu(uz)L(u,z) .
\end{gather*}
From \eqref{U:est5}, \eqref{U:est6}, we get, for $0<|z|\le \frac13R$, $|\arg z|\le \frac32\pi$,
\begin{gather*}
|g(u,z)| \le C_2C_3R u e^{u(\frac23R-q)},\qquad
|h(u,z)| \le C_2C_3 u^2 e^{u(\frac23R-q)} .
\end{gather*}
By \eqref{U:connect}, we can write $\beta_2(u)e^{-\frac12z^2}z^bU(a,b,z^2)$ as the right-hand side of \eqref{1:W2} with $g_2$ replaced by $g_2-g$
and~$h_2$ replaced
by~$h_2-h$. If we choose $q=\frac56R$, $g$ and $h$ become exponentially small as $u\to\infty$, and the theorem is proved.
\end{proof}

\begin{Lemma}\label{U:l5}
Suppose $\Re b\ge 1$.
For all $N=1,2,3,\dots$, we have, as $0<u\to\infty$,
\begin{gather}\label{U:newbeta2}
 \frac{\beta_2(u)2^bu^{1-b}}{\Gamma(1+a-b)}=1+O\left(\frac{1}{u^{2N}}\right).
\end{gather}
Moreover, for all $b\in\C$ and all $N=1,2,3,\dots$, we have, as $0<u\to\infty$,
\begin{gather}\label{U:gammaquotient}
 \frac{\Gamma(1+a-b)}{\Gamma(a)} 2^{2-2b}u^{2b-2}=1+2(1-b)\sum_{s=0}^{N-1} \frac{B_s'(0)}{u^{2s+2}}+O\left(\frac{1}{u^{2N+2}}\right) .
\end{gather}
\end{Lemma}
\begin{proof}
We set
\begin{gather*} T(u,z):=\beta_2(u)e^{-\frac12z^2}z^bU\big(a,b,z^2\big).\end{gather*}
Using \cite[(13.2.12)]{NIST}
\begin{gather*}%\label{U:ancontU}
U\big(a,b,xe^{2i\pi}\big)=e^{-2\pi i b}U(a,b,x)+\frac{2\pi i e^{-\pi i b}}{\Gamma(b)\Gamma(1+a-b)} M(a,b,x)
\end{gather*}
and \eqref{M:W3} we obtain
\begin{gather*}%\label{U:ancontV}
T(u,ze^{i\pi})-e^{-\pi i b}T(u,z)=\beta_2(u)\frac{\pi i 2^bu^{1-b}}{\Gamma(1+a-b)}W_3(u,z) .
\end{gather*}
Now we argue as in the proof of Lemma~\ref{2:l2} (applying Lemma~\ref{U:l4} twice) and arrive at~\eqref{U:newbeta2}.
If $\Re b\ge 1$ the asymptotic formula~\eqref{U:gammaquotient} follows from~\eqref{U:newbeta2} and Lemma~\ref{U:l1}. If $\Re b<1$ we use~\eqref{5:reciprocal}.
\end{proof}

\begin{Theorem}\label{U:t1}
Suppose that $b\in\C$, $N\ge 1$ and $R>0$. Then we can write
\begin{gather}
 \Gamma\big(1+\tfrac14u^2-\tfrac12 b\big)2^{-b}u^{b-1}e^{-\frac12z^2}z^b U\big(\tfrac14u^2+\tfrac12 b,b,z^2\big)\nonumber \\
\qquad {} = zK_{b-1}(uz)\left(\sum_{s=0}^{N-1} \frac{A_s(z)}{u^{2s}}+g_2(u,z)\right)
 -\frac{z}{u}K_b(uz)\left(\sum_{s=0}^{N-1} \frac{B_s(z)}{u^{2s}}+zh_2(u,z)\right),\label{U:asy2}
\end{gather}
where
\begin{gather}\label{U:asy2b}
|g_2(u,z)|+|h_2(u,z)|\le \frac{K_2}{u^{2N}}\qquad\text{for} \quad 0<|z|\le R, \quad u\ge u_2,
\end{gather}
and $K_2$, $u_2$ are constants independent of~$z$ and~$u$. There is no restriction on $\arg z$. The polyno\-mials~$A_s(z)$,~$B_s(z)$ appearing in~\eqref{U:asy2}
are determined by the recursion~\eqref{1:eq3},~\eqref{1:eq4} with $f(z)=z^2$ and the conditions $A_0(z)=1$, $A_s(0)=0$ for $s\ge 1$.

Alternatively, we have
\begin{gather}
 \Gamma\big(\tfrac14u^2+\tfrac12 b\big)2^{b-2}u^{1-b}e^{-\frac12z^2}z^b U\big(\tfrac14u^2+\tfrac12 b,b,z^2\big)\nonumber \\
\qquad{} = zK_{b-1}(uz)\left(\sum_{s=0}^{N-1} \frac{a_s(z)}{u^{2s}}+g_2(u,z)\right)
-\frac{z}{u}K_b(uz)\left(\sum_{s=0}^{N-1} \frac{b_s(z)}{u^{2s}}+zh_2(u,z)\right),\label{U:asy3}
\end{gather}
where again \eqref{U:asy2b} holds. The polynomials $a_s(z)$, $b_s(z)$ are defined by~\eqref{5:a},~\eqref{5:b}.
\end{Theorem}

\begin{proof}
We denote
\begin{gather*} V(u,\mu,z):=\Gamma(1+a-b)2^{-b}u^{b-1}e^{-\frac12z^2}z^b U\big(a,b,z^2\big).\end{gather*}
Then we have
\begin{gather*}%\label{U:V}
 V(u,-\mu,z)=\Gamma(a)2^{b-2}u^{1-b}e^{-\tfrac12 z^2} z^b U\big(a,b,z^2\big)
\end{gather*}
which follows from
\cite[(13.2.40)]{NIST}
\begin{gather*}%\label{U:Kummer}
U(a,b,x)=x^{1-b}U(1+a-b,2-b,x).
\end{gather*}
For any $b\in\C$, \eqref{5:formal1}, \eqref{5:formal2}, \eqref{5:formal3}, \eqref{5:formal4}, \eqref{U:gammaquotient}
show that the expansions~\eqref{U:asy2} and~\eqref{U:asy3} are equivalent.
We will prove~\eqref{U:asy2} and~\eqref{U:asy3} for $\Re \mu\ge 0$ and $\Re \mu<0$, respectively.

Suppose $\Re \mu\ge 0$. Then~\eqref{U:asy2},~\eqref{U:asy2b}
follow from Lemmas~\ref{U:l4} and~\ref{U:l5} when $|\arg z|\le \frac32 \pi-\delta$.
Since the function~$V(u,\mu,z)$ is independent of $R$ we can replace $\frac13 R$ by~$R$.
By Theorem~\ref{3:t1}, we can remove the restriction on $\arg z$.
Note that in the proof of Theorem~\ref{3:t1} we only used that~$W_2(u,z)$ solves~\eqref{1:eq1} and admits the asymptotic expansions~\eqref{1:W2},~\eqref{1:W2a}. Therefore, we can apply
the theorem to the function $V(u,\mu,z)$ in place of $W_2(u,z)$.

Now suppose that $\Re \mu<0$.
Then, using the expansion we just proved,
\begin{gather*}
 V(u,-\mu,z) = zK_{-\mu}(uz)\left(\sum_{s=0}^{N-1} \frac{A_s(-\mu,z)}{u^{2s}}+g_2(u,z)\right)\nonumber \\
\hphantom{V(u,-\mu,z) =}{} -\frac{z}{u}K_{-\mu+1}(uz)\left(\sum_{s=0}^{N-1} \frac{B_s(-\mu,z)}{u^{2s}}+zh_2(u,z)\right).%\label{U:asy5}
\end{gather*}
Using \eqref{5:a}, \eqref{5:b}, \eqref{6:eq7} and $K_\nu(x)=K_{-\nu}(x)$, we obtain \eqref{U:asy3}, \eqref{U:asy2b}.
\end{proof}

So far we considered only asymptotic expansions of $U(a,b,z^2)$ as $0<u\to\infty$.
Now we set $u=te^{i\theta}$, where $t>0$ and $-\frac12\pi<\theta<\frac12\pi$.
Using the notation of Section~\ref{extension}, we have
\begin{gather*}%\label{U:connect2}
 e^{-i\theta}W_2(t,x)=\beta_1(u) e^{-\frac12z^2}z^b M\big(a,b,z^2\big)+\beta_2(u) e^{-\frac12z^2}z^b U\big(a,b,z^2\big) .
\end{gather*}
It is easy to see that Lemma~\ref{U:l1} remains valid.
Since we allow $-\frac12\pi<\theta<\frac12\pi$, $a=\frac14u^2+\frac12b$ may have negative real part. We need a modif\/ication of Lemma~\ref{U:l2}.

\begin{Lemma}\label{U:l6}
Let $b\in\C$, $-\pi<\arg x<0$, $|\arg (a-1)|\le \pi-\delta$ for some $\delta>0$. Then there is a~constant~$Q$ independent of~$a$ such that
\begin{gather*} |\Gamma(a)U(a,b,x)|\le Q .\end{gather*}
\end{Lemma}

\begin{proof}
We use the integral representation \cite[(13.4.14)]{NIST}
\begin{gather*} \big(e^{2\pi i(a-1)}-1\big)\Gamma(a)U(a,b,x)=\int_C e^{-xt}t^{a-1}(1+t)^{b-a-1} dt ,
\end{gather*}
where the contour $C$ starts at $+\infty i$ and follows the positive
imaginary axis, then describes a loop around $0$ in positive direction and returns to~$+\infty i$.
The argument of $t$ starts at~$\frac12\pi$ and increases to~$\frac52\pi$.
It will be suf\/f\/icient to estimate $\Gamma(a)U(a,b,x)$ in the sector $\frac12\pi\le \arg(a-1)\le \alpha_0$, where $\frac12\pi<\alpha_0<\pi$.
The loop is chosen so that $w=\frac{t}{1+t}$ describes the circle $|w|=\cos\theta_0$, where $\theta_0\in(0,\frac12\pi)$ is the unique solution of
the equation
\begin{gather*} \cos\theta_0=e^{\theta_0\tan\alpha_0} .
\end{gather*}
Then one obtains $|w^{a-1}|\le 1$ on the contour $C$ which implies the desired estimate.
\end{proof}

The proofs of Lemma~\ref{U:l5} and Theorem~\ref{U:t1} can be easily modif\/ied to give the desired asymptotic expansions for $u=te^{i\theta}$ as $0<t\to\infty$
for f\/ixed $\theta\in(-\frac12\pi,\frac12\pi)$.
In~\eqref{U:asy2b} we now have $u=te^{i\theta}$, $t\ge t_2$ and $0<|z|\le R$.

\section{Comparison with Temme \cite{T}}\label{Temme}

It is known \cite[(5.11.13)]{NIST} that, as $z\to\infty$, $|\arg z|\le \pi-\delta$,
\begin{gather}\label{T:eq1}
\frac{\Gamma(z+r)}{\Gamma(z+s)}\sim z^{r-s}\sum_{n=0}^\infty \binom{r-s}{n} B_n^{(r-s+1)}(r) \frac{1}{z^n},
\end{gather}
where the generalized Bernoulli polynomials $B_n^{(\ell)}(x)$ are def\/ined by the Maclaurin expansion
\begin{gather*} \left(\frac{t}{e^t-1}\right)^\ell e^{xt}=\sum_{n=0}^\infty B_n^{(\ell)}(x)\frac{t^n}{n!} .\end{gather*}
We apply~\eqref{T:eq1} with $z=\frac14u^2$, $0<u\to\infty$, and $r=1-\frac12 b$, $s=\frac12 b$. Then we obtain with $a=\frac14 u^2+\frac12 b$,
\begin{gather*}%\label{T:eq2}
 \frac{\Gamma(1+a-b)}{\Gamma(a)}2^{2-2b}u^{2b-2}\sim \sum_{n=0}^\infty \frac{d_n}{u^{2n}},
\end{gather*}
where
\begin{gather*} d_n=4^n\binom{1-b}{n}B_n^{(2-b)}\left(1-\frac12 b\right) .\end{gather*}
We notice that
\begin{gather*} \left(\frac{t}{e^t-1}\right)^{2-b}e^{(1-\frac12b)t}=\left(\frac{\frac12t}{\sinh \frac12 t}\right)^{2-b} \end{gather*}
is an even function of $t$.
Therefore, $d_n=0$ for odd $n$.

It follows from \eqref{U:gammaquotient} that
\begin{gather*}%\label{T:eq3}
B_n'(0)=\frac12 \frac{1}{1-b}d_{n+1} ,
\end{gather*}
and then from \eqref{5:a}, \eqref{5:b}
\begin{gather}\label{T:eq4}
a_n(0)=\tilde d_n,\qquad b_n'(0)=-\frac12 \frac1{1-b} \tilde d_{n+1} ,
\end{gather}
where $\tilde d_n$ is obtained from $d_n$ by replacing $b$ by $2-b$, that is,
\begin{gather*}
\tilde d_n= 4^n\binom{b-1}{n}B_n^{(b)}\left(\frac12 b\right) .
\end{gather*}

Temme \cite[(3.22)]{T} obtained the asymptotic expansion of \eqref{U:asy3} involving polyno\-mials~$a^\dagger_n(z)$, $b^\dagger_n(z)$ in place of $a_n(z)$, $b_n(z)$.
The polynomials $a_n^\dagger(z)$, $b_n^\dagger(z)$ as follows. Introduce the function
\begin{gather*}
f(s,z)=e^{z^2\mu(s)} \left(\frac{\frac12 s}{\sinh \frac12s}\right)^b,\qquad \mu(s)=\frac1s-\frac{1}{e^s-1}-\frac12 ,
\end{gather*}
and its Maclaurin expansion
\begin{gather*}%\label{T:f}
f(s,z)=\sum_{k=0}^\infty c_k(z) s^k .
\end{gather*}
Then recursively, set $c_k^{(0)}=c_k$ and
\begin{gather*}
 c_k^{(n+1)}=4\big(z^2c_{k+2}^{(n)}+(1-b+k)c_{k+1}^{(n)}\big),
\end{gather*}
where $k\ge0 $ and $n\ge 0$.
Then set
\begin{gather*} a_n^\dagger=c_0^{(n)},\qquad b_n^\dagger =-2zc_1^{(n)}.\end{gather*}

\begin{Theorem}\label{T:t1}
For every $n=0,1,2,\dots$, we have $a_n=a_n^\dagger$ and $b_n=b_n^\dagger$.
\end{Theorem}

\begin{proof}
The function $f$ satisf\/ies the partial dif\/ferential equation
\begin{gather*} 4\frac{\partial f}{\partial s} =\frac{\partial^2f}{\partial z^2}+\frac1z\left(2b-1-4\frac{z^2}{s}\right)\frac{\partial f}{\partial z}-z^2 f.\end{gather*}
This implies
\begin{gather}\label{T:c}
 4(k+1)c_{k+1}+4zc_{k+1}'=c_k''+\frac{2b-1}{z} c_k'-z^2 c_k, \qquad ()'=\frac{d}{dz}.
\end{gather}
By induction on $n$ one can show that~\eqref{T:c} is also true with $c_k$ replaced by $c_k^{(n)}$ for any $n=0,1,2,\dots$.
If we use this extended equation with $k=0$ and $k=1$, then we obtain~\eqref{1:eq3},~\eqref{1:eq4} with $a_s^\dagger$, $b_s^\dagger$ in place of $A_s$, $B_s$, respectively.

When $z=0$, we have
\begin{gather*} a_n^\dagger(0)=c_0^{(n)}(0)=4^n(1-b)_n c_n(0)=4^n\frac{(1-b)_n}{n!} B_n^{(b)}\left(\frac12b\right) .\end{gather*}
Comparing with~\eqref{T:eq4} and using that $c_n(0)=0$ for odd $n$, we f\/ind, for all $n$,
\begin{gather}\label{T:eq5}
a_n^\dagger(0)=a_n(0).
\end{gather}
Since both $a_n$, $b_n$ and $a_n^\dagger, b_n^\dagger$ solve \eqref{1:eq3}, \eqref{1:eq4}, \eqref{T:eq5} implies that $a^\dagger_n=a_n$, $b^\dagger_n=b_n$ for all $n$.
\end{proof}

\section{Concluding remark}
In this paper we started from Olver's paper \cite{O}, added some results, and then applied them to the conf\/luent hypergeometric functions.
A referee pointed out that Chapter~12 of Olver's book~\cite{O2} contains a reworked version of~\cite{O} also involving error bounds.
It would be interesting to start from this book chapter and derive results analogous to the ones obtained in the present paper.
However, in contrast to~\cite{O} the book chapter assumes that $\mu$ is positive while in our original problem~\cite{C}~$\mu$ is complex.
Therefore, an extension of the results in~\cite[Chapter~12]{O2} to complex~$\mu$ would be required to obtain results for the conf\/luent hypergeometric
functions in full generality.

\pdfbookmark[1]{References}{ref}
\LastPageEnding

\end{document}